\documentclass[letter,12pt]{article}
\usepackage{extsizes}
\usepackage{blindtext}
\usepackage[utf8]{inputenc}
\usepackage{amsmath}
\usepackage{amssymb}
\usepackage[square,sort,comma,numbers]{natbib}
\usepackage{amsthm}
\usepackage{pst-node}
\usepackage{blindtext}
\usepackage{geometry}
 \usepackage{float}
 \usepackage{amssymb}
\usepackage{tikz-cd}
\usepackage{amscd}
\usepackage[all,color]{xy}
\usepackage{sseq}
\usepackage{amsfonts}
\usepackage{enumerate}
\usepackage{graphicx}
\usepackage{fancyhdr}
\usepackage{tikz}
\usepackage{graphics}
\usepackage{hyperref}

\theoremstyle{plain} \numberwithin{equation}{section}

\newtheorem{theorem}{Theorem}[section]

\newtheorem{lemma}[theorem]{Lemma}
\newtheorem{claim}[theorem]{Claim}
\newtheorem{proposition}[theorem]{Proposition}
\theoremstyle{definition}

\newtheorem{remark}[theorem]{Remark}

\newcommand{\pa}{\partial}
\newcommand{\x}{\times}

\def \f-{f^{-1}}
\def \fp-{f^{-1}_\partial}

\usepackage{hyperref}

\title{Tight contact structures on some families of small Seifert fiber spaces}
\author{Shunyu Wan}
\date{}
\usepackage{natbib}
\usepackage{graphicx}

\begin{document}

\maketitle

\begin{abstract}

\small{Suppose $K$ is a knot in a 3-manifold $Y$, and that $Y$ admits a pair of distinct contact structures. Assume that $K$ has Legendrian representatives in each of these contact structures, such that the corresponding Thurston-Bennequin framings are equivalent. This paper provides a method to prove that the contact structures resulting from Legendrian surgery along these two representatives remain distinct. Applying this method to the situation where the starting manifold is $-\Sigma(2,3,6m+1)$ and the knot is a singular fiber, together with convex surface theory
 we can classify the tight contact structures on certain families of Seifert fiber spaces. }
\end{abstract}

\section{Introduction} 

Classifying (non-isotopic positive) tight contact structures on 3-manifolds (especially Seifert fiber spaces) has been studied for a while. Usually one can get an upper bound on the number of such structures from convex surface theory but it can be difficult to actually construct different ones. We first observe that if we start with two different contact structures on some manifold $Y$ whose Ozsv\'ath–Szab\'o contact invariants \citep{HondaKazezMaticContactClassinHeegaardFloer} \citep{OzsvathSzaboHeegaardFloerandContact} are different then we can produce non-isotopic contact structures on a new 3-manifold obtained by surgery on a knot in $Y$. Note this is not a new theorem, as it follows easily from the literature for example the work by Ghiggini \citep{GhigginiOSinvariantsandfillabilityofcontactstructures}, but we have not seen it explicitly stated. 

\begin{theorem} {\label{Legendrian Surgery preserve uniquness}}
Let $\xi_1$ and $\xi_2$ be two contact structures on a 3-manifold $Y$. Take any smooth knot $K$ in Y. Let $K_i$, $i=1, 2$ be a Legendrian representative of $K$ in $\xi_i$. We denote $(Y_i^-,\xi_i^-)$ to be the contact 3-manifold obtained by taking the $-1$ contact surgery (Legendrian surgery) on $(Y,\xi_i)$ respectively. Moreover, we write $c(\xi_i),c(\xi_i^-)$ to be the contact invariant of the correspond contact manifold. Then 
    \begin{equation}
        c(\xi_1)\neq c(\xi_2) \text{ implies } c(\xi_1^-)\neq c(\xi_2^-)
    \end{equation} 
\end{theorem}

A version of the above Theorem is used in the beautiful paper by Ghiggini and Van Horn-Morris \citep{GhigginiVanHornMorrisTightContactStructuresOnBrieskornSpheres}, where they obtained all the contact structure on the ``wrongly''-oriented Brieskorn homology sphere $-\Sigma(2,3,6m-1)$ by doing Legendrian surgery, starting from different contact structures on the 3-manifold $Y$ given by 0-surgery on the right handed trefoil. The way they distinguish the contact structures is by using the contact invariant, however those different structures on $Y$ all have vanishing contact invariant. Thus they need to use twisted coefficients, and by explicitly analysing the contact structures and open book decomposition, they show those vanishing ordinary contact invariants are all non-zero and distinct with twisted coefficients---and after doing Legendrian surgery on those different contact structures on $Y$, the resulting ones on $-\Sigma(2,3,6m-1)$ all have distinct untwisted contact invariant. 
 
 Theorem \ref{Legendrian Surgery preserve uniquness} can be thought of as showing that ``the property of having distinct contact invariants is preserved under Legendrian surgery.'' 
 
 %gives an affirmative answer to the following more general (any 3-manifolds any contact structure) and easier (the contact structure have different contact invariant)  questions: Are Legendrian surgeries preserve the distinction of contact invariant?  

Using the main Theorem together with some 3-manifolds whose tight contact structures are classified, along with some convex surface theory, we are able to further classify tight contact structures on other 3-manifolds. For example, as starting manifolds we can consider the family of Brieskorn spheres $-\Sigma(2,3,6m+1)$, whose tight contact structures have been classified by Tosun. 

\begin{theorem}{\label{BulentClassification}}
(B.Tosun \citep{TosunTightsmallSeifertfiberedmanifolds}) For $m\geq 1$, Brieskorn sphere $-\Sigma(2,3,6m+1)$ has exactly $\frac{m(m+1)}{2}$ tight contact structures which are all strongly fillable and with pairwise distinct non-zero contact invariant. %at least $n$ of them are Stein fillable and and at least $\lfloor n \rfloor$ are not Stein fillable. 
\end{theorem}

%Although the tight contact structure on bunch of families of Seifert fiber space have been classified, the above Theorems are the only known ones for Brieskorn spheres upon to my knowledge. This paper just classifies another family of Brieskorn Spheres. 

By applying Theorem \ref{Legendrian Surgery preserve uniquness} to the case where $Y=-\Sigma(2,3,6m+1)$, and $K$ is a singular fiber, we are able to conclude the following classification Theorem. 

\begin{theorem}{\label{Classifcation of SFS1}}
For $1\leq m$ and $1\leq n < 18m+4$ the Seifert fibered 3-manifolds
$M(-2;\frac{3m+n-2}{6m+2n-5},\frac{2}{3},\frac{5m+1}{6m+1})$ have exactly  $$\sum_{a=1}^{m} (n+3(a-1))a=\frac{(2m+n-2)(m+1)m}{2}$$
different tight contact structures which are all strongly fillable.
\end{theorem}

Moreover, similar to Theorem \ref{Classifcation of SFS1} by using another singular fiber we can get another family of Seifert fiber space whose tight contact structures we can classify. 

\begin{theorem} \label{Classifcation of SFS2}
For $1\leq m$ and $1\leq n < 12m+3$ the Seifert fibered 3-manifolds 
$M(-2;\frac{1}{2},\frac{4(m-1)+2n+1}{6(m-1)+3n+1},\frac{5m+1}{6m+1})$ have exactly   
$$\sum_{a=1}^{m} (n+2(a-1))a=\frac{(4m-4+3n)(m+1)m}{6}$$ different tight contact structures which are all strongly fillable.
\end{theorem}

In these theorems, the notation $M(e_0; r_1, r_2, r_3)$ indicates a Seifert fibered 3-manifold with three exceptional fibers, and having the indicated Seifert invariants. Note that Tosun has classified tight contact structures on some other small Seifert fiber spaces with $e_0=-2$ (Theorem 1.1 \citep{TosunTightsmallSeifertfiberedmanifolds}), but the two families mentioned in above theorem are not included in those results. 

We will first see the proof of Theorem \ref{Legendrian Surgery preserve uniquness} in the next section. In section 3 we will see how to get those Seifert fiber space mentioned in Theorem  \ref{Classifcation of SFS1}  from $-\Sigma(2,3,6m+1)$, and how to use  Theorem \ref{Legendrian Surgery preserve uniquness} together with another lemma due to Jonathan Simone \citep{SimoneTightcontactstrucutresonPlumbedManifolds} to obtain a lower bound on the number of different tight contact structures. Last in section 4 we show that number is equal to an upper bound coming from convex surface theory.

\textbf{Acknowledgments.} The author would like to thank his advisor Tom Mark for his patience and numerous helpful guidance and suggestions. The author would also like to thank B\"ulent Tosun for useful suggestions of the draft. The author was supported in part by grants from the NSF (RTG grant DMS-1839968) and the Simons Foundation (grants 523795 and 961391 to Thomas Mark).

\section{Proof of Theorem 1.1}

The proof is actually quite simple just by manipulating with the naturality and vanishing result of contact invariant under the Legendrian surgery 

\begin{proof}[Proof of Theorem \ref{Legendrian Surgery preserve uniquness}]
The Theorem is clearly true if $Y_1^-$ is not smoothly diffeomorphic to $Y_2^-$, so we suppose $Y_1^-=Y_2^-=Y^-$ in other words we suppose the smooth surgeries are the same. Denote $c(\xi_i)$ and $c(\xi_i^-)$ to be the contact invariant in $\widehat{HF}(-Y)$ and $\widehat{HF}(-Y^-)$ correspond to $\xi_i$ and $\xi_i^-$, then by the hypothesis $c(\xi_1)\neq c(\xi_2)$ we may assume $c(\xi_1)\neq 0$.
Let $W$ be the cobordism from $Y$ to $Y^-$  induced by the smooth surgery on the knot, and let $\overline{W}$ be the opposite cobordism from $-Y^-$ to $-Y$. Legendrian surgery induces a map \citep{OzsvathSzaboHeegaardFloerandContact} $F_{\overline{W}}: \widehat{HF}(-Y^-) \rightarrow \widehat{HF}(-Y)$. Denote by $\mathfrak{t}_i$  the canonical $Spin^c$ structure on $\overline{W}$ induced by the Weinstein structure coming from thinking of $W$ as a Weinstein handle attached to $K_i$ in $\xi_i$, respectively. By \citep[Lemma 2.11]{GhigginiOSinvariantsandfillabilityofcontactstructures}, for any $Spin^c$ structure $\mathfrak{s}$ on $\overline{W}$ we have 
\begin{equation}
    F_{\overline{W},s}(c(\xi_i^-))= \begin{cases} 
      c(\xi) & if \ s=t_i \\
      0 & if \ s\neq t_i.
   \end{cases} 
\end{equation}
We will argue by contradiction that $c(\xi_1^-)\neq c(\xi_2^-)$. Suppose $c(\xi_1^-)=c(\xi_2^-)$. Since smoothly the two Legendrian surgeries are the same surgery on the same knot, they induce the same map on Heegaard Floer homology. However there are two possibilities of the relation between their corresponding $Spin^c$ structures, either $\mathfrak{t}_1=\mathfrak{t}_2$ or not. 

If $\mathfrak{t}_1=\mathfrak{t}_2$ then by equation (2.1) we have $c(\xi_1)=F_{\overline{W},\mathfrak{t}_1}(c(\xi_1^-))=F_{\overline{W},\mathfrak{t}_2}(c(\xi_2^-))=c(\xi_2)$ which is a contradiction. On the other hand if $\mathfrak{t}_1\neq\mathfrak{t}_2$ then again by (2.1) we have $c(\xi_1)=F_{\overline{W},\mathfrak{t}_1}(c(\xi_1^-))=F_{\overline{W},\mathfrak{t}_1}(c(\xi_2^-))=0$ which is a contradiction again. Therefore $c(\xi_1^-)\neq c(\xi_2^-)$. 
\end{proof}

\section{Construction of Seifert fiber spaces and contact structures}

Now we want to use the Theorem \ref{Legendrian Surgery preserve uniquness} in a specific example, but before that we need to see how tight contact structures on those Seifert spaces mentioned in Theorem \ref{Classifcation of SFS1} come from the tight contact structures on $-\Sigma(2,3,6n+1)$ in more detail.

Consider the diagram of Figure \ref{P1} (without the red part) for $M_m=-\Sigma(2,3,6m+1)$. We may also  describe $M_m$ as the Seifert manifold $M(0; \frac{1}{2},-\frac{1}{3},-\frac{m}{6m+1})$; in particular we can decompose it as $(\Sigma \times S^1) \cup_{A_1\cup A_2 \cup A_3} (V_1\cup V_2 \cup V_3)$ where $\Sigma$ is a pair of pants (i.e. 3 punctured sphere) and $V_i$'s are solid torus neighborhood of the singular fibers $F_i$. The $A_i$ are homeomorphisms from $\pa V_i$ to $-\pa(\Sigma \times S^1)$. We also choose identifications on $V_i\cong D^2 \x S^1$, $\pa V_i \cong \mathbb{R}^2/\mathbb{Z}^2$ such that $(1,0)^T$ corresponds to meridian, and $-\pa (M_m\backslash V_i) \cong \mathbb{R}^2/\mathbb{Z}^2$ such that $(0,1)^T$ corresponds to $S^1$ fiber and $(1,0)^T$ corresponds to $-\pa(pt \x \Sigma)$. After we have determined the identifications we can represent the maps $A_i$ using matrices as follows

\[
A_1=\left( \begin{array}{cc}
2&1\\ -1&0
\end{array} \right), \qquad A_2=\left( \begin{array}{cc}
3 & -1 \\ 1&0
\end{array} \right), \qquad A_3=\left( \begin{array}{cc}
6m+1 & 6\\ m & 1 
\end{array} \right).
\] 

\begin{figure} [h]
    \centering
    \includegraphics[width=5cm]{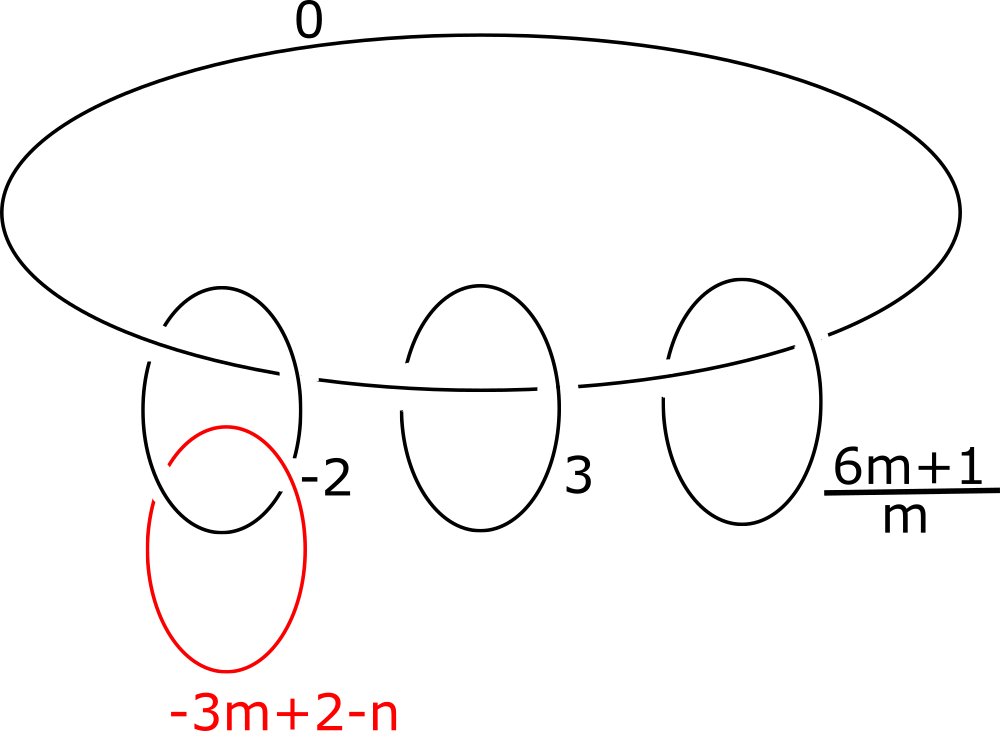}
    \label{P1}
    \caption{The surgery diagram without red circle and red framing describe $M_m=-\Sigma(2,3,6m+1)$, and describe $M_m^n=M(0; \frac{3m+n-2}{6m+2n-5},-\frac{1}{3},-\frac{m}{6m+1} )$ when we consider the red circle and red framing.}
\end{figure}

 We observe that the singular fiber $F_1$ is   represented by the red circle in Figure \ref{P1}, which is the meridian of the $-2$ framed circle in the figure. The Seifert space $M_m^n=M(0; \frac{3m+n-2}{6m+2n-5},-\frac{1}{3},-\frac{m}{6m+1})$ is obtained by taking the surgery along the red circle with the diagram framing $-3m+2-n$. It's not hard to see that $M_m^n$ is equivalent to the Seifert fibered 3-manifold $M(-2;\frac{3m+n-2}{6m+2n-5},\frac{2}{3},\frac{5m+1}{6m+1})$ in Theorem \ref{Classifcation of SFS1} by manipulating the Seifert invariants. Next let's see how the framing on the red circle relates to the different contact structures on $M_m$. 

One way to think about those different contact structures of $M_m$ is by putting them into a triangle with $m$ rows, where there are $a$-th many vertices at the $a$-th row (from top to bottom), and each vertex represents a contact structure on $M_m$. If a vertex is in the $a$-th row then the corresponding contact structure has the maximal twisting number of regular fiber equal to $-6(m-a)-1$ (\citep[Claim 4.2]{TosunTightsmallSeifertfiberedmanifolds}). Given those different maximal twisting numbers of regular fiber together with the attaching map we are able to determine the maximal twisting number of singular fiber corresponding to each row.

In particular if we put the $V_1$ into a standard contact neighborhood with boundary slope $\frac{1}{n_1}$, then the boundary slope on $-\pa(M_m\backslash V_1)$ is $s_1=-\frac{n_1}{2n_1+1}$. If the maximal twisting number of regular fiber is $-6(m-a)-1$ we have $2n_1+1=-6(m-a)-1$, so $n_1=-3(m-a)-1$

To put this another way, we denote twisting number of $F_1$ to be $tw(F_1)=contact \ framing - torus \ framing$ where the torus framing is determined by the longitude $(0,1)^T$ in $V_1$. Then on the $a$-th row of the triangle we have that the maximal twisting number of regular fiber is $-6(m-a)-1$ and the maximal twisting number of $F_1$ is $tw(F_1)=-3(m-a)-1=-3m+3a-1$. 

Moreover, since $A_1 (0,1)^T=(1,0)^T$ we can see that the  framing of $F_1$ corresponding to the choice of $V_1$ exactly corresponds to the 0 framing of the meridian of $-2$ framed circle in the diagram. Hence if we start with a contact structure on $M_m$ lying in the $a$-th row of the triangle, the $(-3m+2-n)$ framing on the red circle in the diagram corresponds to $(-3m+2-n)$ torus framed surgery, which by the twisting number calculation is the smooth coefficient corresponding to $(-3a+3-n)$ contact surgery  (for example, starting with the contact structure the top row, we are considering contact $-n$ surgery). Observe that performing contact $(-3a+3-n)$ surgery involves $(3a-3+n)$ different choices of stabilization to become a Legendrian surgery. Thus each tight contact structure having twisting number $-6(m-a)-1$ gives rise to $(3a-3+n)$ (possibly) different Legendrian surgeries. There are $a$-th many different tight structures for each $1\leq a\leq m$ so there are potentially a total of $\sum_{a=1}^m (n+3(a-1))a$ different tight contact structures on $M_m^n$ given by Legendrian surgery on the singular fiber $F_1$. In summary, we have the following table. 

\begin{center}
 \addtolength{\leftskip} {-2cm}
    \addtolength{\rightskip}{-2cm}
\begin{tabular}{||c c c c c||} 
 \hline
 tight contact  & maximal twisting  & maximal twisting & contact   & different ways    \\ 
 structure & number of  &number of  & surgery & of stabilizing  \\
on $M_m$ & regular fiber & $F_1$ & coefficient & contact surgery
 
 $\xi$\\ [0.5ex] 
 \hline\hline
 $\xi_1^1$ & $-6(m-1)-1$ & $-3m+3-1$ & $-n$ & $n$ \\ 
 $\xi_2^1$ \ \ $\xi_2^2$ & $-6(m-2)-1$ & $-3m+6-1$ & $-n-3$ & $n+3$ \\
 \vdots & \vdots &\vdots & \vdots & \vdots \\
  $\xi_m^1$ \ \ $\xi_m^2$ \ldots $\xi_m^m$ & $-1$ &$-1$ & $-n-3(m-1)$ & $n+3(m-1)$ \\
[1ex] 
 \hline
\end{tabular}
\end{center}

To show all of those resulting contact structure on $M_m^n$ are all different we need Theorem \ref{Legendrian Surgery preserve uniquness} and the following Lemma.  
\begin{lemma}{\label{lm 2.1}}
(J. Simone \citep{SimoneTightcontactstrucutresonPlumbedManifolds})
If $(Y,\xi)$ is weakly symplectically fillable and $(W.J_i)$ is a Stein cobordism from $(Y,\xi)$ to $(Y',\xi_i)$ for $i=1,2$ such that the $Spin^c$ structures induced by $J_1$ and $J_2$ are not isomorphic, then $\xi_1$ and $\xi_2$ are non-isotopic tight contact structures.
\end{lemma}

Now we are able to show the following Proposition.

\begin{proposition} {\label{prop}}
All tight contact structures on $M_m^n$ mentioned above are non-isotopic and strongly fillable. 
\end{proposition}

\begin{proof}
First by Theorem \ref{BulentClassification} all tight contact structures on $M_m$ are strongly fillable. Since different stabilization have different rotation number which implies the $Spin^c$ structures induced by Legendrian surgeries are different, it follows from Lemma \ref{lm 2.1} that for any contact structure $\xi_i^j$ if we choose different stabilization, Legendrian surgery will result in different contact structures. Moreover, since Legendrian surgery preserves strong fillability, the contact structures we obtained are all strongly fillable.

Furthermore, for two different contact structures $\xi_i^j$ and $\xi_p^q$ their contact invariants are different and non-zero by Theorem \ref{BulentClassification}, and by construction we are considering the same smooth surgery on $F_1$. By Theorem \ref{Legendrian Surgery preserve uniquness} any pair of Legendrian surgeries on smoothly isotopic knots in $\xi_i^j$ and $\xi_p^q$ will give different contact structure on $M_m^n$. Therefore none of the tight contact structures on $M_m^n$ mentioned above are isotopic. 
\end{proof}

\section{Upper Bound}
Our way to determine an upper bound on the number of  distinct tight contact structures is a typical application of convex surface theory. 

\begin{proposition}
$M_m^n$ has at least $\sum_{a=1}^m (3a-3+n)a$ different tight contact structure when $1\leq m$ and $0\leq n < 18m+3$
\end{proposition}

\begin{proof}
The proof is presented assuming the reader is familiar with convex surface theory  \citep{HondaOntheClassificationOfTightContactStructures}. The argument is parallel to the ones in \citep{GhigginiSchonenbergerOntheclassificationoftightcontactstructures},\citep{GhigginiVanHornMorrisTightContactStructuresOnBrieskornSpheres} and \citep{TosunTightsmallSeifertfiberedmanifolds}.

We first rewrite the Seifert notation for $M_m^n$ as $M(0; -\frac{1}{3},-\frac{m}{6m+1},\frac{3m+n-2}{6m+2n-5} )$ (we exchange the fiber order here). As before we can describe $M_m^n$ as $M_m^n=(\Sigma \times S^1) \cup_{A_1\cup A_2 \cup A_3} (V_1\cup V_2 \cup V_3)$ where $\Sigma$ is a pair of pants (i.e. 3 punctured sphere) and $V_i$'s are solid torus neighborhood of the singular fibers $F_i$. Again the $A_i$ are maps from $\pa V_i$ to $-\pa(\Sigma \times S^1)$. We again choose identifications on $V_i\cong D^2 \x S^1$, $\pa V_i \cong \mathbb{R}^2/\mathbb{Z}^2$ such that $(1,0)^T$ correspond to the meridian, and $-\pa M_m^n\backslash V_i \cong \mathbb{R}^2/\mathbb{Z}^2$ such that $(0,1)^T$ corresponds to the $S^1$ fiber and $(1,0)^T$ corresponds to the $-\pa(pt \x \Sigma)$. Then we can choose 

\[
A_1=\left( \begin{array}{cc}
3&-1\\ 1&0
\end{array} \right), \qquad A_2=\left( \begin{array}{cc}
6m+1 & 6 \\ m&1
\end{array} \right), \qquad A_3=\left( \begin{array}{cc}
6m+2n-5 & 2\\ -3m-n+2 & -1 
\end{array} \right),
\]
we also have the inverse of the $A_i$
 
\[ 
A_1^{-1}=\left( \begin{array}{cc}
0&1\\ -1&3
\end{array} \right), \qquad A_2^{-1}=\left( \begin{array}{cc}
1&-6 \\ -m&6m+1
\end{array} \right), \qquad A_3^{-1}=\left( \begin{array}{cc}
-1 & -2\\ 3m+n-2 & 6m+2n-5 
\end{array} \right).
\]

We would like to first determine the maximal twisting number of tight contact structure on $M_m^n$. 

\begin{claim}
If $\xi$ is a tight contact structure on $M_m^n$ then $tw(\xi) \geq -6m+5$
\end{claim}

\begin{proof}
As a standard starting point of arguments like this, we begin with very negative twisting number $n_i$ on $V_i$ so the boundary slope on $V_i$ are $\frac{1}{n_i}$ and when measured on $-\pa (M_m^n\backslash V_i)$ it becomes $$s_1=\frac{n_1}{3n_1-1}, \ s_2=\frac{mn_2+1}{(6m+1)n_2+6}, \ s_3=\frac{(-3m-n+2)n_3-1}{(6m+2n-5)n_3+2}$$

Note when $3n_1-1\neq(6m+1)n_2 +6$, and taking the Legendrian ruling slope on each $\pa (M_m^n\backslash V_i)$ to be infinite using Giroux flexibility, if we consider the vertical annulus $\mathcal{A}$  whose Legendrian boundaries are Legendrian rulings on $\pa (M_m^n\backslash V_1)$ and $\pa (M_m^n\backslash V_2)$,  the Imbalance Principle \citep{HondaOntheClassificationOfTightContactStructures} implies the existence of a bypass on one side or the other of $\mathcal{A}$---and then the Twist Number Lemma \citep{HondaOntheClassificationOfTightContactStructures} allows us to increase the the corresponding twisting number by one. Iterating the procedure, so long as the assumption $3n_1-1\neq(6m+1)n_2 +6$ continues to hold then the Twist Number Lemma will apply to allow us to increase $n_1$ up to $-2m+2$ and $n_2$ to $-1$. If this happens, then at this point $3n_1-1=(6m+1)n_2 +6=-6m+5$. Hence if we let $tw(\xi)=t$ we need to show that if we increase $n_1$ and $n_2$ to a point at which $3n_1-1=(6m+1)n_2 +6=t$, then $t\geq -6m+5$. 

Assume $3n_1-1=(6m+1)n_2 +6$ and write this common value as $t$. Observe that this implies that there is some integer $k$ with $n_1=-(6m+1)k-(2m-2)$, and $n_2=-3k-1$, and $t=-3(6m+1)k-(6m-5)$. Since $n_1,n_2 <0$, we have $k\geq 0$. If $k = 0$ then we are done, so we suppose $k\geq 1$.  We consider the annulus $\mathcal{A}$ again: since we assume that we have achieved the twisting number, the dividing set $\mathcal{A}$ consists of parallel arcs connecting the boundary components of $\mathcal{A}$. If we cut along $\mathcal{A}$  and round the edge we get a smooth manifold $ M_m^n\backslash (V_1 \cup V_2 \cup \mathcal{A} )$ such that the boundary is smoothly isotopic to $\pa (M_m^n\backslash V_3)$. Moreover, by the Edge Rounding lemma \citep{HondaOntheClassificationOfTightContactStructures} the slope is 

$$s_{\pa(M_m^n\backslash (V_1 \cup V_2 \cup \mathcal{A}))}=\frac{n_1}{3n_1-1}+\frac{mn_2+1}{(6m+1)n_2+6}-\frac{1}{3n_1-1}=\frac{(9m+1)k+(3m-2)}{3(6m+1)k+(6m-5)}.$$ 

By applying $A_3^{-1}$ to the corresponding vector, we find that when measured in $\pa V_3$ the slope is  $$s_{\pa V_3}= -\frac{(12m+n-1)k-n}{k-1}=-(12m+n-1)-\frac{12m-1}{k-1}$$.

Now if $k=1$, we have $t=-3(6m+1)-(6m-5)=-24m+2$. Since $1\leq m$, we see $t\leq -22$. On the other hand by the above calculation, when $k=1$ the slope on $\pa V_3$ is infinite and is $-\frac{1}{2}$ on $-\pa (M_m^n \backslash V_3)$ which implies the maximal twisting number is at least $-2$. This contradicts the maximality of t. 

If $k\geq 2$, we have $t \leq -6(6m+1)-(6m-5)=-42m-1$. Using the above slope again and since $1\leq m$ and $1\leq n$, we have $s_{\pa V_3} < -12$. Moreover, since we start with boundary slope $\frac{1}{n_3}$ with $n_3$ very negative, by using  \citep[Corollary 4.8]{HondaOntheClassificationOfTightContactStructures} we can find a torus (between a small standard torus around $F_3$ and the torus obtained above by cutting and rounding) with slope $-1$. This is $\frac{3m+n-3}{-6m-2n+7}$ when measured on $-\pa (M_m^n \backslash V_3)$. However if $n< 18m+4$ we have $t \leq -42m-1 < -6m-2n+7$ contradicting the maximality of $t$.
\end{proof}

It was proved by Wu \citep{WuLegendrianVerticalCircles} for Seifert space with $e_0=-2$ the maximal twisting number is less than $0$. So we left to analyze the situation when $-6m+5\leq tw(\xi)\leq -1$ 

We have already arranged that $n_2=-1$ and $n_1=-2m+2$. If there is no bypass on either $V_1$ or $V_2$ we are in the above situation where $k=0$, i.e. $tw(\xi)=-6m+5$. Since $V_1$, $V_2$ are standard neighborhoods there is unique tight contact structure on $V_1$ and $V_2$. Moreover, the tight contact structure on neighborhood of $\mathcal{A}$ is uniquely determined by the dividing set which are just horizontal arcs. Finally, the slope on $\pa V_3$ is $-n$ (strictly, this is the slope on the boundary of $M_m^n\backslash (V_1 \cup V_2 \cup \mathcal{A})$, viewed from $V_3$). By the classification of tight contact structures on solid torus (Theorem 2.3 \citep{HondaOntheClassificationOfTightContactStructures})  there are exactly $n$ different tight contact structure on $V_3$. Hence together there are at most $n$ different tight contact structure on $M_m^n$ when $tw(\xi)=-6m+5$.

If there exist a bypass on  $V_1$ by the Twist Number lemma we can increase $n_1$ to $-2m+3$, and then consider the dividing set on $\mathcal{A}$. Again by the Imbalance Principle there exists a bypass on $-\pa(M_m^n \backslash V_2)$ with ruling slope $\infty$. Since $s_2=\frac{-m+1}{-6m-5}=\frac{m-1}{6m-5}$, after attaching bypass and applying  \citep[Lemma 3.15]{HondaOntheClassificationOfTightContactStructures} the new boundary slope become $s_2=\frac{(m-1)-1}{6(m-1)-5}=\frac{m-2}{6m-11}$. Using the Imbalance Principle and the Twist Number Lemma again we increase $n_1$ to $-2m+4$. (Same situation happens if there exists a bypass on $V_2$) 

If there is no further bypass on $V_1$ or $V_2$ we are in the situation where $tw(\xi)=-6(m-1)+5$. We cut and round edges to get $$s_{\pa(M_m^n\backslash (V_1 \cup V_2 \cup \mathcal{A}))}=\frac{n_1}{3n_1-1}+\frac{(m-1)-1}{-(6m-1)+5}-\frac{1}{3n_1-1}=\frac{3m-5}{6m-11}$$ and again by applying $A_3^{-1}$ we get the slope $-3-n$ on $\pa V_3$. Moreover, since $s_2=\frac{m-2}{6m-11}$  by applying $A_2^{-1}$ we got slope $-2$ on $\pa V_2$. Therefore since there is $1$ tight contact structure on $V_1$ and $\mathcal{A}$, $2$ tight contact structure on $V_2$ and $3+n$ on $V_3$ there are at most $2\times (3+n)$ tight contact structure on $M_m^n$ with $tw(\xi)=-6m+11$.

Inductively we can do this until $n_1=0$. i.e. we have $(m-1)$ times of choosing the existence of a bypass on $V_1$ (or $V_2$). Suppose we can do it till the $l$-th time, which means that at the $l$-th step we increase $n_1$ from $-2m+2l$ to $-2m+2l+2$, the slope on $s_2$ changes from $\frac{(m-(l-1))-1}{6(m-(l-1))-5}$ to $\frac{(m-l)-1}{6(m-l)-5}$ and the maximal twisting number is $tw(\xi)=-6(m-l)+5$. Cut, round, and apply $A_3^{-1}$ again: we find slope $-3l-1-n$ on $\pa V_3$. Applying $A_2^{-1}$ gives slope $-l-1$ on $\pa V_2$. Hence there are at most $(3l+1+n)(l+1)$ tight contact structures on $M_m^n$ with $tw(\xi)=-6(m-l)+5$.

In summary we have the following table: 

\begin{center}
\begin{tabular}{||c c c c||} 
 \hline
 $tw(\xi)$ & $n_1$ & slope on $\pa V_3$  & Slope on $\pa V_2$   \\ 
 [0.5ex] 
 \hline\hline
 $-6m+5$ & $-2m+2$ & $-n$ & $-1$  \\ 
 $-6(m-1)+5$  & $-2m+4$ & $-n-3$ & $-2$  \\
 \vdots & \vdots &\vdots & \vdots  \\
 $-1$ & $0$ & $-n-3(m-1)$ & $-m$  \\
[1ex] 
 \hline
\end{tabular}
\end{center}

Therefore in total we have at most $\sum_{a=1}^m (3a-3+n)a=\frac{(2m+n-2)(m+1)m}{2}$ for $1\leq n<18m+4$.

\end{proof}

An exactly parallel argument works for Theorem 1.4; we just do the surgery along the singular fiber $F_2$ in $-\Sigma (2,3,6m+1)$ instead of $F_1$. We obtain another family of Seifert fiber spaces as stated in the Theorem, and analogous proofs follow for determining the contact structures.   
\begin{remark}
    The above types of argument could potentially also be used for Legendrian surgery on the third singular fiber, and other 3-manifolds with contact structures having different contact invariants (For example $-\Sigma(2,3,6n-1)$. However the framing calculation and the convex surface argument is more subtle, so the author chooses to not present them here. 
\end{remark}

\bibliographystyle{plain}
\bibliography{bib}

\end{document}